\documentclass{amsart}
\usepackage{amsfonts}
\usepackage{graphicx}

\usepackage{bbm}
\usepackage{amscd}
 \usepackage{amsmath}

\usepackage{amssymb}
\usepackage{mathrsfs}
 \usepackage{epsfig} 
 \usepackage{graphics}
\newtheorem*{mainthm}{Main Theorem}
\newtheorem{theorem}{Theorem}[section]
\newtheorem{cor}[theorem]{Corollary}
\newtheorem{lem}[theorem]{Lemma}

\theoremstyle{definition}
\newtheorem{defn}[theorem]{Definition}
\theoremstyle{remark}
\newtheorem{rem}[theorem]{Remark}
\newtheorem{example}[theorem]{Example}

\numberwithin{equation}{section}

\begin{document}

\title[Time-preserving Structural Stability of Differential Systems]
{Time-preserving Structural Stability of Hyperbolic Differential Dynamics with Noncompact Phase Spaces}%
\author[X.~Dai]{Xiongping Dai}%
\address{Department of Mathematics\\ Nanjing University\\
Nanjing, 210093, P. R. CHINA}%
\email{xpdai@nju.edu.cn}%

\thanks{This project was supported by
NSFC (No.~10671088) and 973 project (No.~2006CB805903)}%

\date{May 17, 2008}
\begin{abstract}
Let $S\colon\mathbb{E}\rightarrow\mathbb{R}^n$ where
$T_w\mathbb{E}=\mathbb{R}^n$ for all $w\in\mathbb{E}$, be a
$C^1$-differential system on an $n$-dimensional Euclidean $w$-space
$\mathbb{E}$, which naturally gives rise to a flow
$\phi\colon(t,w)\mapsto t_\cdot w$ on $\mathbb{E}$, and let
$\Lambda$ be a $\phi$-invariant closed subset containing no any
singularities of $S$. If $\Lambda$ is compact and hyperbolic, then
Anosov's theorem asserts that $S$ is structurally stable on
$\Lambda$ in the sense of topological equivalence; that is, for any
$C^1$-perturbation $V$ close to $S$, there is an
$\varepsilon$-homeomorphism $H\colon \Lambda\rightarrow \Lambda_V$
sending orbits $\phi(\mathbb{R}, w)$ of $S$ into orbits
$\phi_V(\mathbb{R},H(w))$ of $V$ for all $w$ in $\Lambda$. In this
paper, using Liao theory Anosov's result is generalized as follows:
Let $\psi_V\colon\mathbb{R}\times\Sigma\rightarrow\Sigma$ be the
cross-section flow of $V$ relative to $S$ locally defined on the
Poincar\'{e} cross-section bundle
$\Sigma=\bigcup_{w\in\Lambda}\Sigma_w$ of $S$, where
$\Sigma_w=\left\{w^\prime\in\mathbb{E}\,|\,\langle
S(w),w^\prime-w\rangle=0\right\}$. If $S$ is hyperbolic on $\Lambda$
and $V$ is $C^1$-close to $S$, then there is an
$\varepsilon$-homeomorphism $w\mapsto H(w)\in\Sigma_w$ from
$\Lambda$ onto a closed set $\Lambda_V$ such that
$\psi_V(t,H(w))=H(t_\cdot w)$ for all $w\in\Lambda$, where $\Lambda$
need not be compact. Finally, an example is provided to illustrate
our theoretical outcome.
\end{abstract}

\maketitle
\section{Introduction}\label{sec1}
In~\cite{L66,L74}, professor S.-T.~Liao established the theory of
standard systems of differential equations for $C^1$-differential
dynamical systems on compact Riemannian manifolds. Then he
systematically applied methods in the qualitative theory of ODE to
study stability problems of differentiable dynamical systems via his
theory~\cite{L96}. We in~\cite{Dai,Dai07} generalized in part Liao's
theory to differential systems on Euclidean spaces. Via the
generalized, in turn we can apply the approaches of ergodic theory
and differentiable dynamical systems to the study of the qualitative
theory of ODE~\cite{Dai07,Dai08}. In the present paper, we continue
to perfect Liao theory and give a further application.

Assume, throughout this paper, that
$S\colon\mathbb{E}\rightarrow\mathbb{R}^n$ is a $C^1$-vector field
on an $n$-dimensional Euclidean $w$-space $\mathbb{E}$, where $n\ge
2$ and $T_w\mathbb{E}=\mathbb{R}^n$ for all $w$, and the equation
$\dot{w}=S(w)$ naturally induces a continuous-time dynamical system
$\phi\colon\mathbb{R}\times\mathbb{E}\rightarrow\mathbb{E};\
(t,w)\mapsto t_\cdot w$ on the phase-space $\mathbb{E}$. Let
$$
\Sigma={\bigcup}_{w\in\mathbb{E}}\Sigma_w\quad\textrm{ where }
\Sigma_w=\left\{w^\prime\in\mathbb{E}\,|\,\langle
S(w),w^\prime-w\rangle=0\right\},
$$
be the cross-section bundle of $S$. Then, $S$ gives naturally rise
to a formal (local) Poincar\'{e} cross-section flow
$$
\psi\colon\mathbb{R}\times\Sigma\rightarrow\Sigma;\ (t,w+x)\mapsto
t_\cdot w +\psi_{t,w}x,
$$
where $w^\prime=w+x$ means $w^\prime\in\Sigma_w$ and where
$\psi_{t,w}\colon\Sigma_{w}-w\rightarrow\Sigma_{t_\cdot w}-t_\cdot
w$ is locally well defined for any
$(t,w)\in\mathbb{R}\times\mathbb{E}$ by
$\psi_{t,w}x=\phi(t_0,w+x)-t_\cdot w$, where $t_0$ is the first
$t^\prime>0$ when $t>0$ or the first $t^\prime<0$ when $t<0$ with
$\phi(t^\prime,w+x)\in\Sigma_{t_\cdot w}$. Clearly, $\psi$ is a
local skew-product flow based on $\phi$ satisfying
$\psi_{t,w}\textbf{0}=\textbf{0}$.

Let $\mathfrak{X}^1(\mathbb{E})$ be the space of all $C^1$-vector
fields on $\mathbb{E}$ endowed with the $C^1$-topology induced by
the usual $C^1$-norm $\|\cdot\|_1$. Then, for any
$V\in\mathfrak{X}^1(\mathbb{E})$, on $\Sigma$ we may also naturally
define a formal local skew-product flow
$$
\psi_V\colon\mathbb{R}\times\Sigma\rightarrow\Sigma;\ (t,w+x)\mapsto
t_\cdot w+\psi_{V;t,w}x.
$$
Note here that $\psi_{V;t,w}\textbf{0}$ need not equal $\textbf{0}$
when $V\not=S$.

Let $\mathbb{T}_w=T_\textbf{0}\Sigma_w$ be the $(n-1)$-dimensional
tangent space to the hyperplane $\Sigma_w$ at $w+\textbf{0}$ for all
$w\in\mathbb{E}$ and
$\mathbb{T}=\bigcup_{w\in\mathbb{E}}\mathbb{T}_w$ called the
\emph{transversal tangent bundle to $S$} over $\mathbb{E}$. Clearly,
$\mathbb{T}_w=\Sigma_w-w=\{x\in\mathbb{R}^n\,|\,\langle
S(w),x\rangle=0\}$. Then, we can define naturally the linear
skew-product flow transversal to $S$
$$
\Psi\colon\mathbb{R}\times\mathbb{T}\rightarrow\mathbb{T};\
(t,(w,x))\mapsto(t_\cdot w,\Psi_{t,w}x),
$$
where $\Psi_{t,w}\colon\mathbb{T}_{w}\rightarrow\mathbb{T}_{t_\cdot
w}$ is defined as $\Psi_{t,w}=D_\textbf{0}\psi_{t,w}$ for any
$(t,w)\in\mathbb{R}\times\mathbb{E}$, associated with $S$.

Recall that a $\phi$-invariant closed subset $\Lambda$ is said to be
\emph{hyperbolic}, provided that there exist constants $C\ge 1,
\lambda<0$ and a continuous $\Psi$-invariant splitting
\begin{equation*}
\mathbb{T}_{w}=T_{w}^s\oplus T_{w}^u\quad w\in \Lambda
\end{equation*}
such that
\begin{align*}
\|\Psi_{t_0+t,w}x\|&\le C^{-1}\exp(\lambda
t)\|\Psi_{t_0,w}x\|\quad\forall\,x\in
T_{w}^s\\
\intertext{and} \|\Psi_{t_0+t,w}x\|&\ge C\exp(-\lambda
t)\|\Psi_{t_0,w}x\|\quad\forall\,x\in T_{w}^u
\end{align*}
for any $t_0\in\mathbb{R}$ and for all $t>0$.

Then, Anosov's structural stability theorem~\cite{Ano,Mos} asserts
that: \textit{If $\Lambda$ is a compact hyperbolic set for $S$, then
for any $\varepsilon>0$ there is a $C^1$-neighborhood $\mathcal {U}$
of $S$ in $\mathfrak{X}^1(\mathbb{E})$ such that, if $V\in\mathcal
{U}$ then there exists a $\varepsilon$-topological mapping $h$ from
$\Lambda$ onto some subset $\Lambda_V$ of $\mathbb{E}$ which sends
orbits of $S$ in $\Lambda$ into orbits of $V$ in $\Lambda_V$.}

This important theorem was extended to axiom A differential systems
\cite{PS,Ro}, and to $C^0$-perturbations by considering the
so-called semi-structural stability independently by~\cite{KM,L74};
for discrete versions, see~\cite{Rob,Robs,Mos,Wal, Nit}. On another
direction, in this paper, we study the structural stability of
\emph{noncompact} hyperbolic set under \emph{time-preserving}
conjugacy between the induced cross-section flows. More precisely,
using Liao theory we prove the following.

\begin{mainthm}
Let $\Lambda$ be a hyperbolic set for $S$, not necessarily compact,
satisfying the following conditions:
\begin{enumerate}
\item[(U1)] The first derivative $S^\prime(w)$ is uniformly bounded on $\Lambda$;

\item[(U2)] $0<\inf_{w\in\Lambda} \|S(w)\|\le\sup_{w\in\Lambda}
\|S(w)\|<\infty$;

\item[(U3)] $S^\prime(w)$ is uniformly continuous at $\Lambda$; that is, to any
$\epsilon>0$ there is some $\delta>0$ so that for any $\mathfrak
w\in\Lambda$, $\|S^\prime(w)-S^\prime(\mathfrak w)\|<\epsilon$
whenever $\|w-\mathfrak w\|<\delta$.
\end{enumerate}
Then, for any $\varepsilon>0$ there is a $C^1$-neighborhood
$\mathcal {U}$ of $S$ in $\mathfrak{X}^1(\mathbb{E})$ such that for
any $V\in\mathcal {U}$ there exists a $\varepsilon$-topological
mapping $H$ from $\Lambda$ onto some closed subset $\Lambda_V$ which
sends orbits of $S$ in $\Lambda$ into orbits of $V$ in $\Lambda_V$,
such that $H(w)\in \Sigma_w$ and $\psi_V(t,H(w))=H(t_\cdot w)$ for
all $w\in\Lambda$ and for any $t\in\mathbb{R}$.
\end{mainthm}

Notice here that if $\Lambda$ is compact, then conditions (U1), (U2)
and (U3) hold automatically. So our result is an extension of the
classical one. Even for the compact case, the time-preserving
property is still a new ingredient in our main theorem.

To prove this result, we will introduce the reduced standard systems
of differential equations for perturbations of $S$ in $\S\ref{sec2}$
and recall Liao's exponential dichotomy in $\S\ref{sec3}$. Finally,
we prove the Main Theorem in $\S\ref{sec4}$, using simplified and
extended Liao approach that is completely different from Anosov's
geometrical approach~\cite{Ano} and Moser's functional
approach~\cite{Mos} to differentiable dynamical systems on compact
Riemannian manifolds. And in $\S\ref{sec4}$, we will construct a
differential system which has a noncompact, structurally stable,
hyperbolic subset.

\section{Liao standard system of differential equations}\label{sec2}

Let $S$ be any given $C^1$-differential system on $\mathbb{E}$ and
$\Lambda$ a $\phi$-invariant closed subset in $\mathbb{E}$
satisfying conditions (U1), (U2) and (U3) as in the Main Theorem
stated in $\S\ref{sec1}$. Around a regular orbit
$\phi(\mathbb{R},w)$ we defined in \cite{Dai07} the reduced standard
systems for $S$ itself. However, we will introduce below the
standard systems for perturbations $V$ of $S$.

\subsection{}

As usual in Liao theory~\cite{Dai07,Dai08}, let $\mathscr
F_{n-1}^{*\sharp}(\Lambda)={\bigsqcup}_{w\in\Lambda}\mathscr
F_{n-1,w}^{*\sharp}$ be the bundle of transversal orthonormal
$(n-1)$-frames, where the fiber over $w$ is defined as
$$
\mathscr
F_{n-1,w}^{*\sharp}=\left\{{\gamma}=(\vec{u}_1,\ldots,\vec{u}_{n-1})\in\mathbb{T}_w\times\cdots\times\mathbb{T}_w\,|\langle\vec{u}_i,\vec{u}_j\rangle=\delta^{ij}\textrm{
for }1\le i,j\le n-1 \right\},
$$
endowed with the naturally induced topology. Then, $S$ naturally
generates a skew-product flow over $\phi$
\begin{equation}\label{eq2.1}
\chi^{*\sharp}\colon\mathbb{R}\times\mathscr
F_{n-1}^{*\sharp}(\Lambda)\rightarrow\mathscr
F_{n-1}^{*\sharp}(\Lambda);\ (t,(w,{\gamma}))\mapsto(t_\cdot
w,\chi_{t,w}^{*\sharp}{\gamma}),
\end{equation}
where $\chi_{t,w}^{*\sharp}\colon\mathscr
F_{n-1,w}^{*\sharp}\rightarrow\mathscr F_{n-1,t_\cdot w}^{*\sharp}$
is defined by the standard Gram-Schmidt orthonormalization process;
cf.~\cite{Dai, Dai07} for the details.

Let $\textbf{e}=\{\vec{e}_1,\ldots,\vec{e}_{n-1}\}$ where
$\vec{e}_j=(\stackrel{j^{\textrm{th}}}{0,\ldots,0,1,0,\ldots,0})^\textrm{T}\in\mathbb
R^{n-1}$, be the standard basis of $\mathbb R^{n-1}$ and we view
$y\in\mathbb R^{n-1}$ with components $y^1,\ldots,y^{n-1}$ as a
column vector $(y^1,\ldots,y^{n-1})^\textrm{T}$ and
$\gamma\in\mathscr F_{n-1,w}^{*\sharp}$ as an $n$-by-$(n-1)$ matrix
with columns
$\textrm{col}_1{\gamma},\ldots,\textrm{col}_{n-1}{\gamma}$
successively.

Given any orthonormal $(n-1)$-frame $(w,{\gamma})\in \mathscr
F_{n-1}^{*\sharp}(\Lambda)$, sometimes written simply as $\gamma_w$,
we define by linear extension the linear transformation
\begin{equation}\label{eq2.2}
\mathcal T_{{\gamma}_w}^*\colon\mathbb
R^{n-1}\rightarrow\mathbb{T}_w
\end{equation}
in the way
\begin{equation*}
\vec{e}_j\mapsto {\mathrm{col}}_j{\gamma}\quad(1\le j\le n-1).
\end{equation*}
Since ${\gamma}$ is an orthonormal basis of $\mathbb T_w$, $\mathcal
T_{\gamma_w}^*$ is an isomorphism such that
\begin{equation*}
\mathcal T_{{\gamma}_w}^*(y)={\gamma}
y=\sum_{j=1}^{n-1}y^j\textrm{col}_j\gamma\textrm{ and }
\|y\|_{\mathbb R^{n-1}}= \|{\gamma} y\|_{\mathbb{T}_w}\quad\forall\,
y\in\mathbb R^{n-1}.
\end{equation*}
Moreover, we now define
\begin{equation}\label{eq2.3}
C_{{\gamma}_w}^*(t)=\mathcal
T_{\chi^{*\sharp}(t,{\gamma}_w)}^{*-1}\circ \Psi_{t,w}\circ \mathcal
T_{\gamma_w}^*\quad\forall\, t\in\mathbb R,
\end{equation}
where $\chi^{*\sharp}\colon\mathbb{R}\times\mathscr
F_{n-1}^{*\sharp}(\Lambda)\rightarrow\mathscr
F_{n-1}^{*\sharp}(\Lambda)$ as in (\ref{eq2.1}). Then the
commutativity holds:
\begin{equation}\label{eq2.4}
\begin{CD}
\mathbb R^{n-1}@>{C_{\gamma_w}^*(t)}>>\mathbb R^{n-1}\\
@V{\mathcal T_{{\gamma}_w}^*}VV @VV{\mathcal
T_{\chi^{*\sharp}(t,{\gamma}_w)}^*}V\\
\mathbb T_w@>{\Psi_{t,w}}>>\mathbb T_{t_\cdot w}.
\end{CD}
\end{equation}
We now think of $C_{{\gamma}_w}^*(t)$ as an $(n-1)\times
(n-1)$-matrix under the base $\textbf{e}$ of $\mathbb{R}^{n-1}$.
Clearly, $t\mapsto\frac{d}{dt}C_{{\gamma}_w}^*(t)$ makes sense since
${S}$ is of class $C^1$ and by (\ref{eq2.4}) we have
\begin{equation}\label{eq2.5}
C_{{\gamma}_w}^*(t_1+t_2)=
C_{\chi^{*\sharp}(t_1,{\gamma}_w)}^*(t_2)\circ
C_{{\gamma}_w}^*(t_1)\quad\forall\,t_1,t_2\in\mathbb R.
\end{equation}
Put
\begin{equation}\label{eq2.6}
R_{{\gamma}_w}^*(t)=\left\{\frac{d}{dt}C_{{\gamma}_w}^*(t)\right\}{C_{{\gamma}_w}^*(t)}^{-1}\quad\forall\,(w,{\gamma})\in\mathscr
F_{n-1}^{*\sharp}(\Lambda).
\end{equation}

\begin{defn}\label{def2.1}
The linear differential equation
\begin{equation*}
\dot{y}=R_{{\gamma}_w}^*(t)y\quad(t,y)\in\mathbb R\times\mathbb
R^{n-1} \leqno{(R_{{\gamma}_w}^*)}
\end{equation*}
for any $(w,{\gamma})\in\mathscr F_{n-1}^{*\sharp}(\Lambda)$, is
called the \emph{reduced linearized system} of ${S}$ under the
moving frame $\chi^{*\sharp}(t,{\gamma}_w)$. See \cite{Dai,Dai07}.
\end{defn}

These reduced linearized systems of ${S}$ possess the following
properties.

\begin{lem}[{\cite{Dai,Dai07}}]\label{lem2.2}
The following statements hold:
\begin{enumerate}
\item Uniform boundedness: $R_{{\gamma}_w}^{*}(t)$  is
continuous in $(t,(w,\gamma))$ in
$\mathbb{R}\times\mathscr{F}_{n-1}^{*\sharp}(\Lambda)$ with
$$\eta_\Lambda:=\sup\left\{\sum_{i,j}|R_{{\gamma}_w}^{*ij}(t)|;\, t\in\mathbb{R}, (w,\gamma)\in\mathscr F_{n-1}^{*\sharp}(\Lambda)\right\}<\infty.$$

\item Upper triangularity: ${R}_{\gamma_w}^*(t)$ is upper-triangular with
$$
{R}_{\gamma_w}^*(t)=\left[\begin{matrix}
{\omega}_1^*(\chi^{*\sharp}(t,\gamma_w))&\cdots&*\\
\vdots&\ddots&\vdots\\
0&\cdots&{\omega}_{n-1}^*(\chi^{*\sharp}(t,\gamma_w))
\end{matrix}\right]\quad\forall\,t\in\mathbb R
$$
where ${\omega}_k^*(w,\gamma), 1\le k\le n-1$, called the ``Liao
qualitative functions" of $S$, are uniformly continuous in
$(w,\gamma)\in\mathscr{F}_{n-1}^{*\sharp}(\Lambda)$.

\item Geometrical interpretation: Let $\vec{v}={\gamma}{y}\in
\mathbb T_w$ for ${y}\in\mathbb R^{n-1}$. If $y(t)=y(t,{y})$ is the
solution of $(R_{\gamma_w}^*)$ with $y(0)={y}$, then
$$
\Psi_{t,w}\vec{v}=\mathcal T_{\chi^{*\sharp}(t,\gamma_w)}^*
y(t)=(\chi_{t,w}^{*\sharp}\gamma) y(t).
$$
Conversely, letting $x(t)=(x^1(t),\ldots,x^{n-1}(t))^{\mathrm
T}\in\mathbb R^{n-1}$ be defined by
$$
x^i(t)=\left\langle
\Psi_{t,w}\vec{v},\mathrm{col}_i{\chi_{t,w}^{*\sharp}}\gamma\right\rangle_{t_\cdot
w}\quad i=1,\ldots,{n-1},
$$
we have $\dot{x}(t)={R}_{\gamma_w}^*(t)x(t)\textrm{ and }x(0)={y}$.
Particularly, $C_{{\gamma}_w}^*(t)$ is the fundamental matrix
solution of $(R_{{\gamma}_w}^*)$.
\end{enumerate}
\end{lem}

As a consequence of the above lemma, we have

\begin{cor}\label{cor2.3}
Let $\Lambda$ be hyperbolic for $S$ associated to $\Psi$-invariant
splitting $\mathbb{T}_\Lambda=T_\Lambda^s\oplus T_\Lambda^u$. Then,
there are two constants $\pmb{\eta}>0$ and $\textbf{d}>0$ such that:
for any $(w,\gamma)\in\mathscr{F}_{n-1}^{*\sharp}(\Lambda)$, if
$\mathrm{col}_i\gamma\in T_w^s$ for $i=1,\ldots,\dim T_w^s$ then
\begin{align*}
&\int_0^T\omega_k^*(\chi^{*\sharp}(t_0+t,(w,\gamma)))\,dt\le
-\pmb{\eta}T,\quad 1\le k\le \dim T_w^s\\
\intertext{and}&\int_0^T\omega_k^*(\chi^{*\sharp}(t_0+t,(w,\gamma)))\,dt\ge
\pmb{\eta}T,\quad \dim T_w^s+1\le k\le n-1
\end{align*}
for any $t_0\in\mathbb{R}$ and for all $T\ge \textbf{d}$.
\end{cor}

\begin{proof}
The statement comes immediately from Lemma~\ref{lem2.2} and
\cite[Lemma~3.7]{L74}.
\end{proof}

\subsection{}For a constant $c>0$, let
$\mathbb R_c^{n-1}=\{y\in\mathbb R^{n-1};\,\|y\|<c\}$. Fix any $w\in
\Lambda$. For any $\gamma\in\mathscr F_{n-1,w}^{*\sharp}$, we need
the $C^1$-mapping
\begin{equation*}
\mathcal
P_{w,\gamma}^*\colon\mathbb{R}\times\mathbb{R}^{n-1}\rightarrow\mathbb{E}
\end{equation*}
defined by
$$
\mathcal P_{w,\gamma}^*(t,y)=t_\cdot
w+({\chi_{t,w}^{*\sharp}}\gamma) y\in\Sigma_{t_\cdot
w}\quad\forall\,(t,y)\in\mathbb{R}\times\mathbb{R}^{n-1}.
$$
It is known~\cite[Lemma~5.1]{Dai07} that there is a constant
$\mathfrak c>0$, which is independent of
$(w,\gamma)\in\mathscr{F}_{n-1}^{*\sharp}(\Lambda)$, such that
$\mathcal P_{w,\gamma}^*$ is locally diffeomorphic on $\mathbb
R\times\mathbb R_\mathfrak c^{n-1}$. In fact, according to
\cite{Dai07} there is some $\epsilon>0$ so that for any
$w\in\Lambda$, $\mathcal P_{w,\gamma}^*$ is diffeomorphic from
$(-\epsilon,\epsilon)\times\mathbb{R}_\mathfrak c^{n-1}$ into
$\mathbb{E}$.

Given any $(w,\gamma)\in\mathscr F_{n-1}^{*\sharp}(\Lambda)$. Define
a $C^0$-vector field on $\mathbb{R}\times\mathbb{R}_{\mathfrak
c}^{n-1}$
$$\widehat{S}_{w,{\gamma}}\colon\mathbb{R}\times\mathbb{R}_{\mathfrak
c}^{n-1}\rightarrow\mathbb R^n$$ with
$\widehat{S}_{w,{\gamma}}(t,\textbf{0})=(1,
\textbf{0})^\textrm{T}\in\mathbb{R}\times\mathbb{R}^{n-1}$ in the
following way:
$$
\left(D_{(t,y)}\mathcal
P_{w,\gamma}^*\right)\widehat{S}_{w,{\gamma}}(t,y)=S(\mathcal
P_{w,\gamma}^*(t,y))\quad\forall\,(t,y)\in\mathbb{R}\times\mathbb{R}_{\mathfrak
c}^{n-1}.
$$
Since $\mathcal P_{w,\gamma}^*$ is locally $C^1$-diffeomorphic,
$\widehat{S}_{w,{\gamma}}(t,y)$ is well defined. We now consider the
autonomous system
\begin{subequations}\label{eq2.7}
\begin{align}
\frac{d}{d\mathbbm{t}}\left(\begin{matrix} t\\
y
\end{matrix}\right)&=\widehat{S}_{w,{\gamma}}(t,y)\quad(t,y)\in\mathbb{R}\times\mathbb{R}_{\mathfrak
c}^{n-1}\label{Req2.7a}\\
\intertext{and write}
\widehat{S}_{w,{\gamma}}(t,y)&=\left(\widehat{S}_{w,{\gamma}}^0(t,y),\ldots,\widehat{S}_{w,{\gamma}}^{n-1}(t,y)\right)^{\textrm{T}}\in\mathbb
R\times\mathbb R^{n-1}.\label{eq2.7b}
\end{align}
\end{subequations}
Next, put
\begin{equation}\label{eq2.8}
{S}_{w,{\gamma}}^*(t,y)=\left(\frac{\widehat{S}_{\mathfrak
w,{\gamma}}^1(t,y)}{\widehat{S}_{w,{\gamma}}^0(t,y)},\ldots,\frac{\widehat{S}_{w,{\gamma}}^{n-1}(t,y)}{\widehat{S}_{w,{\gamma}}^0(t,y)}\right)^{\textrm{T}}\in\mathbb
R^{n-1} \quad\forall\,(t,y)\in\mathbb R\times\mathbb R_{\mathfrak
c}^{n-1}.
\end{equation}

\begin{defn}[\cite{Dai07}]\label{def2.4}
The non-autonomous differential equation
\begin{equation*}
\dot{y}={S}_{w,{\gamma}}^*(t,y)\quad(t,y)\in\mathbb R\times\mathbb
R_{\mathfrak c}^{n-1}\leqno{({S}_{w,{\gamma}}^*)}
\end{equation*}
is called the \emph{reduced standard system of $S$ under the base
$(w,{\gamma})\in\mathscr F_{n-1}^{*\sharp}(\Lambda)$}.
\end{defn}

Is is easy to see that
\begin{equation}\label{eq2.9}
{S}_{w,{\gamma}}^*(t+t_1,y)={S}_{\chi^{*\sharp}(t,(w,{\gamma}))}^*(t_1,y)\quad\forall\,(t,y)\in\mathbb
R\times\mathbb R_{\mathfrak c}^{n-1}.
\end{equation}
For convenience of our later discussion, we write
\begin{equation}\label{eq2.10}
\mathcal {P}_{w,\gamma}^*(t,y)=t_\cdot w+\widehat{\mathcal
{P}}_{w,\gamma}^*(t,y),\textrm{ where }\widehat{\mathcal
{P}}_{w,\gamma}^*(t,y)\in\mathbb{T}_{t_\cdot w}.
\end{equation}
The following is important for our later arguments.

\begin{lem}[{\cite{Dai07}}]\label{lem2.5}
Under the conditions (U1), (U2) and (U3), the following statements
hold: for any $(w,\gamma)\in\mathscr{F}_{n-1}^{*\sharp}(\Lambda)$
\begin{enumerate}
\item ${S}_{w,{\gamma}}^*(t,\mathbf{0})=\mathbf{0}\in\mathbb R^{n-1}$ for all $t\in\mathbb R$,
and ${S}_{w,{\gamma}}^*(t,y)$ is continuous with respect to
$(t,y)\in\mathbb R\times\mathbb R_{\mathfrak c}^{n-1}$.

\item For any $(\bar{t},\bar{y})\in\mathbb R\times\mathbb R_{\mathfrak
c}^{n-1}$, let $\bar{w}=\bar{t}_\cdot w+\bar{x}={\mathcal
P}_{w,\gamma}^*(\bar{t},\bar{y})\in\Sigma_{\bar{t}_\cdot w}$ and
$$
y^*(t)=y_{w,{\gamma}}^*(t;\bar{t},\bar{y})\quad
t\in(r^\prime,r^{\prime\prime})\textrm{ where }
\bar{t}\in(r^\prime,r^{\prime\prime}),
$$
be the solution of $({S}_{w,{\gamma}}^*)$ with
$y^*(\bar{t})=\bar{y}$. Then
$$
\psi(t-\bar{t},\bar{w})={\mathcal
P}_{w,\gamma}^*(t,y^*(t))\in\Sigma_{t_\cdot
w}\quad(r^\prime<t<r^{\prime\prime}).
$$

\item ${S}_{w,{\gamma}}^*(t,y)$ is of class $C^1$ with respect to $y\in\mathbb
R_{\mathfrak c}^{n-1}$ such that
$$
{\partial {S}_{w,{\gamma}}^*(t,y)}/{\partial
y}\to{R}_{w,{\gamma}}^*(t)\textrm{ as }y\to\mathbf{0}
$$
uniformly for $(t,(w,\gamma))\in\mathbb
R\times\mathscr{F}_{n-1}^{*\sharp}(\Lambda)$.
\end{enumerate}
\end{lem}

\noindent From here on, for any $w\in \Lambda$ we will rewrite
($S_{w,\gamma}^*$) as
\begin{subequations}\label{eq2.11}
\begin{align}
&\dot{y}=R_{w,\gamma}^*(t)y+S_{\textsl{rem}(w,\gamma)}^*(t,y)\quad(t,y)\in\mathbb R\times\mathbb R_{\mathfrak c}^{n-1}\label{eq2.11a}\\
\intertext{where}&S_{\textsl{rem}(w,\gamma)}^*(t,y)=S_{w,\gamma}^*(t,y)-R_{w,\gamma}^*(t)y.\label{eq2.11b}
\end{align}
\end{subequations}
Then, we have the following result.

\begin{lem}[{\cite{Dai07}}]\label{lem2.6}
Under the conditions (U1), (U2) and (U3), to any $\kappa>0$, there
is some $\xi\in(0,\mathfrak c]$ so that
$$
\|S_{\textsl{rem}(w,\gamma)}^*(t,y)-S_{\textsl{rem}(w,\gamma)}^*(t,y^\prime)\|\le
\kappa\|y-y^\prime\|\quad \textrm{whenever }
y,y^\prime\in\mathbb{R}_\xi^{n-1}
$$
holds uniformly for $(t,(w,\gamma))\in\mathbb
R\times\mathscr{F}_{n-1}^{*\sharp}(\Lambda)$.
\end{lem}

\subsection{}

In what follows, we let $V\colon \mathbb{E}\rightarrow\mathbb{R}^n$
be an arbitrarily given another $C^1$ vector field on $\mathbb{E}$.
Note here that $(\mathscr
F_{n-1}^{*\sharp}(\Lambda),\chi^{*\sharp})$ still corresponds to
$S$.

In order to introduce the standard systems of $V$ associated with
$S$, let us consider firstly a simple lemma.

\begin{lem}\label{lem2.7}
Let $h\colon \widehat{N}\rightarrow N$ be a map of class $C^1$ from
a $C^1$ manifold $\widehat{N}$ into another $C^1$ manifold $N$. Let
$\widehat{X}$ and $X$ be $C^0$ vector fields on $\widehat{N}$ and
$N$, respectively. If $(Dh)\widehat{X}=X$ then for any
$\hat{p}\in\widehat{N}$, $h$ maps the integral curve
$\phi_{\hat{\textsl{x}}}(t,\hat{p})$ of $\widehat{X}$ into an
integral curve $\phi_{\textsl{x}}(t,h(\hat{p}))$ of $X$ such that
$\phi_{\textsl{x}}(t,h(\hat{p}))=h(\phi_{\hat{\textsl{x}}}(t,\hat{p}))$.
\end{lem}

\begin{proof}
Let $h(\hat{p})=p$. Define a $C^1$ curve in $N$ by $C\colon t\mapsto
h(\phi_{\hat{\textsl{x}}}(t,\hat{p}))$. Since
$$\frac{d}{dt}\phi_{\hat{\textsl{x}}}(t,\hat{p})=\widehat{X}(\phi_{\hat{\textsl{x}}}(t,\hat{p}))
\textrm{ and }
(Dh)\widehat{X}(\phi_{\hat{\textsl{x}}}(t,\hat{p}))=X(C(t))=\frac{d}{dt}C(t),
$$
we get that $C(t)$ is an integral curve of $X$ satisfying the
initial condition $C(0)=p$. Now put $\phi_\textsl{x}(t,p)=C(t)$,
which satisfies the requirement of Lemma~\ref{lem2.7}.
\end{proof}

Particularly, we will be interesting to the case where
$\widehat{N}=\mathbb R\times\mathbb R_\mathfrak c^{n-1},
N=\mathbb{E}$ and $h=\mathcal P_{w,\gamma}^*$ and $X=V$ for any
given $(w,\gamma)\in\mathscr F_{n-1}^{*\sharp}(\Lambda)$.
Correspondingly, there $\widehat{X}$ is right the so-called lifting
system that we are going to define.

\begin{defn}\label{def2.8}
Given any $(w,\gamma)\in\mathscr F_{n-1}^{*\sharp}(\Lambda)$. Define
a $C^0$-vector field $$\widehat{V}_{w,\gamma}\colon\mathbb
R\times\mathbb R_\mathfrak c^{n-1}\rightarrow\mathbb R^n$$ in the
following way:
$$
\left(D_{(t,y)}\mathcal
P_{w,\gamma}^*\right)\widehat{V}_{w,\gamma}(t,y)=V(\mathcal
P_{w,{\gamma}}^*(t,y))\quad\forall\,(t,y)\in\mathbb R\times\mathbb
R_\mathfrak c^{n-1}.
$$
Then, the autonomous differential equation
\begin{subequations}\label{eq2.12}
\begin{align}
\frac{d}{d\mathbbm{t}}\left(\begin{matrix} t\\
y
\end{matrix}\right)&=\widehat{V}_{w,{\gamma}}(t,y)\quad\mathbbm{t}\in\mathbb R, (t,y)\in\mathbb
R\times\mathbb R_\mathfrak c^{n-1}\label{eq2.12a}
\end{align}
\end{subequations}
is referred to as a \emph{lifting} of $V$ under the moving frames
$(\chi^{*\sharp}(t,(w,{\gamma})))_{t\in\mathbb R}$.
\end{defn}

Write
$$
\widehat{V}_{w,{\gamma}}(t,y)=\left(\widehat{V}_{w,{\gamma}}^0(t,y),\ldots,\widehat{V}_{w,{\gamma}}^{n-1}(t,y)\right)^{\textrm{T}}\in\mathbb
R\times\mathbb R^{n-1}.
$$
Clearly, it follows from $\mathcal P_{w,\gamma}^*(t,y)=\mathcal
P_{\chi^{*\sharp}(t,(w,\gamma))}^*(0,y)$ that
\begin{equation}\label{eq2.13}
\widehat{V}_{w,{\gamma}}(t,y)=\widehat{V}_{\chi^{*\sharp}(t,(w,{\gamma}))}(0,y)\quad\forall\,(t,y)\in\mathbb
R\times\mathbb R_\mathfrak c^{n-1}.
\end{equation}
Although $\mathcal P_{w,{\gamma}}^*$ is only $C^1$, we can obtain
more about the regularity of $\widehat{V}_{w,{\gamma}}(t,y)$ with
respect to $y\in\mathbb R_\mathfrak c^{n-1}$ as long as $V$ is
$C^1$.

\begin{lem}\label{lem2.9}
Given any $(w,{\gamma})\in\mathscr F_{n-1}^{*\sharp}(\Lambda)$, the
lifting $\widehat{V}_{w,{\gamma}}(t,y)$ is of class $C^1$ in $y$;
precisely, for $1\le i\le {n-1}$,
$\partial\widehat{V}_{w,{\gamma}}(t,y)/{\partial y^i}$ makes sense
and is continuous with respect to $(t,y,(w,{\gamma}))$ in $\mathbb
R\times\mathbb R_\mathfrak c^{n-1}\times\mathscr
F_{n-1}^{*\sharp}(\Lambda)$.
\end{lem}

\begin{proof}
The statement comes immediately from the regularity of $\mathcal
P_{w,\gamma}^*(t,y)$, as the argument of \cite[Lemma~5.3]{Dai07}.
\end{proof}

Next, let
\begin{equation*}
\{S,V\}_{\Lambda}^1=\sup_{\begin{subarray}{c}(t,y)\in\mathbb
R\times\mathbb
R_{\mathfrak c}^{n-1}\\
(w,\gamma)\in\mathscr
F_{n-1}^{*\sharp}(\Lambda)\end{subarray}}\left\{\|\widehat{S}_{w,\gamma}(t,y)-\widehat{V}_{w,\gamma}(t,y)\|
+\|\frac{\partial}{\partial
y}[\widehat{S}_{w,\gamma}(t,y)-\widehat{V}_{w,\gamma}(t,y)]\|\right\}.
\end{equation*}
From (\ref{eq2.13}) we get
\begin{equation*}
\{S,V\}_{\Lambda}^1=\sup_{\begin{subarray}{c}y\in\mathbb
R_{\mathfrak c}^{n-1}\\
(w,\gamma)\in\mathscr
F_{n-1}^{*\sharp}(\Lambda)\end{subarray}}\left\{\|\widehat{S}_{w,\gamma}(0,y)-\widehat{V}_{w,\gamma}(0,y)\|
+\|\frac{\partial}{\partial
y}[\widehat{S}_{w,\gamma}(0,y)-\widehat{V}_{w,\gamma}(0,y)]\|\right\}.
\end{equation*}
Then, we have

\begin{lem}\label{lem2.10}
There exists some constant $\flat_\Lambda>0$ such that
\begin{equation*}
\|S-V\|_1\ge\flat_\Lambda\{S,V\}_\Lambda^1\quad\forall\,V\in\mathfrak
X^1(\mathbb{E}).
\end{equation*}
\end{lem}

\begin{proof}
For any $(w,\gamma)\in\mathscr F_{n-1}^{*\sharp}(\Lambda)$ let
$$
J_{w,\gamma}(y)=\left.\frac{\partial\mathcal
P_{w,\gamma}^*(t,y)}{\partial (t,y)}\right|_{(0,y)}
$$
be the $n$-by-$n$ Jacobi matrix of $\mathcal P_{w,\gamma}^*(t,y)$ at
$(0,y)\in\mathbb{R}\times\mathbb{R}_{\mathfrak c}^{n-1}$. Then
\begin{align*}
\mathcal P_{w,\gamma}^*(0,y)&=w+\gamma y\\
\intertext{and}
J_{w,\gamma}(y)&=\left[S(w)+\left.\frac{d}{dt}\right|_{t=0}({\chi_{t,w}^{*\sharp}}\gamma)y,\,\gamma\right]_{n\times
n}.
\end{align*}
Thus, for any $y\in\mathbb{R}_{\mathfrak c}^{n-1}$ we have
$$
\widehat{S}_{w,\gamma}(0,y)-\widehat{V}_{w,\gamma}(0,y)={J_{w,\gamma}(y)}^{-1}(S-V)(w+\gamma
y)
$$
Moreover, from condition (U1) we can prove by the argument of
\cite[Lemma~5.3]{Dai07} that
$\left.\frac{d}{dt}\right|_{t=0}{\chi_{t,w }^{*\sharp}}\gamma$,
viewed as an $n$-by-$(n-1)$ matrix, is uniformly continuous and
bounded for any $(w,\gamma)\in\mathscr F_{n-1}^{*\sharp}(\Lambda)$.
Therefore, there is some constant $\flat_\Lambda>0$ which satisfies
the requirement of Lemma~\ref{lem2.10}.
\end{proof}

From Lemma~\ref{lem2.10}, condition (U1) and
$\widehat{S}_{w,\gamma}^0(t,\textbf{0})=1$, we may assume, without
any loss of generality replacing $\mathfrak c$ by a more small
positive constant if necessary, that
\begin{itemize}
\item $\exists\,\mathcal N_\textsl{s}$, a $C^1$-neighborhood of
$S$ in $\mathfrak X^1(\mathbb{E})$ such that: for any $V\in\mathcal
N_\textsl{s}$

\item
$\frac{1}{2}\le\widehat{V}_{w,\gamma}^0(t,y)\le\frac{4}{2}$ for any
$(t,y,{(w,\gamma)})\in\mathbb R\times\mathbb R_\mathfrak
c^{n-1}\times\mathscr F_{n-1}^{*\sharp}(\Lambda)$.
\end{itemize}
Thus, the following definition makes sense.

\begin{defn}\label{def2.11}
Given any $(V,(w,\gamma))\in\mathcal N_\textsl{s}\times\mathscr
F_{n-1}^{*\sharp}(\Lambda)$, set
\begin{equation*}
V_{w,\gamma}^*(
t,y)=\left(\frac{\widehat{V}_{w,\gamma}^1(t,y)}{\widehat{V}_{w,\gamma}^0(t,y)},\ldots,\frac{\widehat{V}_{w,\gamma}^{n-1}(
t,y)}{\widehat{V}_{w,\gamma}^0(t,x)}\right)^{\textrm{T}}\in\mathbb
R^{n-1} \quad\forall\,(t,y)\in\mathbb R\times\mathbb R_\mathfrak
c^{n-1}.
\end{equation*}
The non-autonomous differential equation
\begin{equation*}
\dot{y}=V_{w,\gamma}^*(t,y)\quad(t,y)\in\mathbb R\times\mathbb
R_\mathfrak c^{n-1}\leqno{(V_{w,\gamma}^*)}
\end{equation*}
is referred to as the \emph{standard system of $V$ associated to
$(S,(w,\gamma))$}.
\end{defn}

From (\ref{eq2.13}) we have
\begin{equation}\label{eq2.15}
V_{w,\gamma}^*(t+t^\prime,y)=V_{\chi^{*\sharp}(t,(w,\gamma))}^*(t^\prime,y)\quad\forall\,t,t^\prime\in\mathbb
R\textrm{ and }y\in\mathbb R_\mathfrak c^{n-1}.
\end{equation}
In what follows, we write $(V_{w,\gamma}^*)$ as
\begin{equation*}
\dot{y}=R_{w,\gamma}^*(t)y+V_{\textsl{rem}(w,\gamma)}^*(t,y)\quad(t,y)\in\mathbb
R\times\mathbb R_\mathfrak c^{n-1}\leqno{(V_{w,\gamma}^*)}
\end{equation*}
where
\begin{subequations}\label{eq2.16}
\begin{align}
V_{\textsl{rem}(w,{\gamma})}^*(t,y)&=V_{w,\gamma}^*(
t,y)-R_{w,\gamma}^*(t)y\label{eq2.16a}\\
\intertext{such that}
V_{\textsl{rem}(w,{\gamma})}^*(t+t^\prime,y)&=V_{\textsl{rem}(\chi^{*\sharp}(t,(w,\gamma)))}^*(t^\prime,y)\quad\forall\,t,t^\prime\in\mathbb
R.\label{eq2.16b}
\end{align}
\end{subequations}

Similar to Lemma~\ref{lem2.5}, we obtain the following result.

\begin{theorem}\label{thm2.12}
Given any $V\in\mathcal N_\textsl{s}$, the following statements
hold:
\begin{enumerate}
\item $V_{\textsl{rem}(w,{\gamma})}^*(t,y)$ and ${\partial V_{\textsl{rem}(w,\gamma)}^*(t,y)}/{\partial y}$
are continuous in $(t,y,{(w,\gamma)})\in\mathbb R\times\mathbb
R_\mathfrak c^{n-1}\times\mathscr F_{n-1}^{*\sharp}(\Lambda)$.

\item Given any ${(w,\gamma)}\in\mathscr F_{n-1}^{*\sharp}(\Lambda)$. If $y^*(t)=y_{V;w,\gamma}^*(t;t_0,y))$ where $t^\prime<t,t_0<t^{\prime\prime}$, is the solution of
$(V_{w,\gamma}^*)$ with $y^*(t_0)=y$, then
$$
\psi_V(t-t_0,{\mathcal P}_{w,\gamma}^*(t_0,y))={\mathcal
P}_{w,\gamma}^*(t,y^*(t))\in\Sigma_{t_\cdot w}.
$$
\end{enumerate}
\end{theorem}

Moreover, similar to Lemma~\ref{lem2.6} we have the following
important result.

\begin{theorem}\label{thm2.13}
The following three statements hold.
\begin{enumerate}
\item Given any $(V,{(w,\gamma)})\in\mathcal N_\textsl{s}\times\mathscr F_{n-1}^{*\sharp}(\Lambda)$, there is some
$L>0$ such that
$$
\|V_{\textsl{rem}(w,\gamma)}^*(t,y)-V_{\textsl{rem}(w,\gamma)}^*(t,y^\prime)\|\le
L\|y-y^\prime\|
$$
for any $t\in\mathbb R$ and for any $y,y^\prime\in\mathbb
R_\mathfrak c^{n-1}$.

\item To any $\eta>0$ there exists a $C^1$-neighborhood $\mathcal U_\textsl{s}^\prime\subset\mathcal N_\textsl{s}$ of $S$ and $\xi^\prime\in(0,\mathfrak c]$
such that: $\forall\,V\in\mathcal U_\textsl{s}^\prime$
$$
\sup_{(t,y)\in\mathbb R\times\mathbb
R_{\xi^\prime}^{n-1}}\|V_{\textsl{rem}(w,\gamma)}^*(t,y)\|\le
\eta\xi^\prime\quad\forall\,(w,{\gamma})\in\mathscr
F_{n-1}^{*\sharp}(\Lambda).
$$

\item To any given $\kappa>0$ there corresponds a $C^1$-neighborhood $\mathcal U_\textsl{s}^{\prime\prime}\subset\mathcal N_\textsl{s}$ of $S$ and a constant $\xi^{\prime\prime}\in(0,\mathfrak c]$
such that: $\forall\,V\in\mathcal U_\textsl{s}^{\prime\prime}$
$$
\|V_{\textsl{rem}(w,\gamma)}^*(t,y)-V_{\textsl{rem}(w,\gamma)}^*(t,y^\prime)\|\le
\kappa\|y-y^\prime\|\quad\forall\,y,y^\prime\in\mathbb
R_{\xi^{\prime\prime}}^{n-1}
$$
uniformly for $(t,(w,\gamma))\in\mathbb R\times\mathscr
F_{n-1}^{*\sharp}(\Lambda)$.
\end{enumerate}
\end{theorem}

\begin{proof}
By (\ref{eq2.16a}), Lemma~\ref{lem2.10} and Lemma~\ref{lem2.2} and
condition (U1)
$$
L:=\sup_{(t,y)\in\mathbb R\times\mathbb R_\mathfrak
c^{n-1}}\left\{\|\partial V_{w,\gamma}^*(t,y)/\partial
y\|+\|R_{w,\gamma}^*(t)\|\right\}<+\infty
$$
which satisfies the requirement of the statement (1).

Given any $\eta>0$. For any $V\in\mathcal N_\textsl{s}$ and for any
$(w,\gamma)\in\mathscr F_{n-1}^{*\sharp}(\Lambda)$ one can write
$$
V_{\textsl{rem}(w,\gamma)}^*(t,y)=\left(V_{w,\gamma}^*(t,y)-{S}_{w,\gamma}^*(t,y)\right)+\left({S}_{w,\gamma}^*(t,y)-{R}_{w,\gamma}^*(t)y\right)
$$
for any $(t,y)\in\mathbb R\times\mathbb R_\mathfrak c^{n-1}$. Then,
from Lemma~\ref{lem2.10} and Lemma~\ref{lem2.6} there exists a
$C^1$-neighborhood $\mathcal U_\textsl{s}^\prime\subset\mathcal
N_\textsl{s}$ of $S$ and a constant $\xi^\prime\in(0,\mathfrak c]$
such that
$$
\sup_{(t,y)\in\mathbb R\times\mathbb
R_{\xi^\prime}^{n-1}}\|V_{w,\gamma}^*(t,y)\|\le
\eta\xi^\prime\quad\forall\,(w,{\gamma})\in\mathscr
F_{n-1}^{*\sharp}(\Lambda)\textrm{ and }V\in\mathcal
U_\textsl{s}^\prime.
$$
This shows the statement (2).

Now given any $\kappa>0$. Next, for any $V\in\mathcal N_\textsl{s}$
consider
$$
\frac{\partial}{\partial
y}V_{\textsl{rem}(w,\gamma)}^*(t,y)=\frac{\partial}{\partial
y}\left(V_{w,\gamma}^*(t,y)-{S}_{w,\gamma}^*(t,y)\right)+\left(\frac{\partial}{\partial
y}{S}_{w,\gamma}^*(t,y)-{R}_{w,\gamma}^*(t)\right).
$$
From Lemma~\ref{lem2.10} we obtain that
$$
\|\frac{\partial}{\partial
y}\left({V}_{w,\gamma}^*(t,y)-{S}_{w,\gamma}^*(t,y)\right)\|\to
0\textrm{ as } \|V-S\|_1\to0
$$
uniformly for $(t,y,(w,{\gamma}))\in\mathbb R\times\mathbb
R_\mathfrak c^{n-1}\times\mathscr F_{n-1}^{*\sharp}(\Lambda)$ and,
from Lemma~\ref{lem2.2} there exists
$\xi^{\prime\prime}\in(0,\mathfrak c]$ so that
$$
\|\frac{\partial}{\partial
y}{S}_{w,\gamma}^*(t,y)-{R}_{w,\gamma}^*(t)\|\le\frac{\kappa}{2}\quad\forall\,(t,(w,\gamma))\in\mathbb
R\times\mathscr F_{n-1}^{*\sharp}(\Lambda)\textrm{ and }y\in\mathbb
R_{\xi^{\prime\prime}}^{n-1}.
$$
Hence, there is a $C^1$-neighborhood $\mathcal
U_\textsl{s}^{\prime\prime}\subset\mathcal N_\textsl{s}$ of $S$ such
that: $\forall\,V\in\mathcal U_\textsl{s}^{\prime\prime}$
$$
\|\frac{\partial}{\partial
y}V_{\textsl{rem}(w,\gamma)}^*(t,y)\|\le\kappa\quad\forall\,(t,(w,\gamma))\in\mathbb
R\times\mathscr F_{n-1}^{*\sharp}(\Lambda)\textrm{ and }y\in\mathbb
R_{\xi^{\prime\prime}}^{n-1}.
$$
This implies the statement (3) by the mean value theorem.

Thus, Theorem~\ref{thm2.13} is proved.
\end{proof}

\section{Exponential dichotomy}\label{sec3}

In this section, we will introduce the exponential dichotomy due to
Liao~\cite{L74}, by which we consider in part the relationship
between the phase portraits of linear differential equations and
their small perturbations on Euclidean spaces. Here we shall deal
with families of ordinary differential equations, nor only a single
equations.

Given a positive integer $p$. For convenience of our later
discussion, let $M_{p\times p}^\vartriangle$ be the set of
continuous matrix-valued functions
$A\colon\mathbb{R}\rightarrow\textrm{gl}(p,\mathbb{R})$ such that
\begin{enumerate}
\item[(a)] $A(t)$ is triangular with $A_{ij}(t)=0$ for $1\le j<i\le
p$;

\item[(b)] $A$ is uniformly bounded on $\mathbb{R}$ with
$\eta_A:=\sup_{t\in\mathbb{R}}\|A(t)\|<\infty$;

\item[(c)] $A$ is hyperbolic with index $p_-$ in the following
sense:
$$
\xi_A:=\sup_{t\in\mathbb{R}}
\left\{\sum_{k=1}^{p_-}\int_{-\infty}^te^{\int_s^tA_{kk}(\tau)\,d\tau}\,ds+
\sum_{k=1+p_-}^{p}\int_t^\infty
e^{\int_s^tA_{kk}(\tau)\,d\tau}\,ds\right\}<\infty.
$$
\end{enumerate}
In addition, let $M_{p\times 1}$ be the set of continuous functions
$f\colon\mathbb{R}\times\mathbb{R}^p\rightarrow\mathbb{R}^p$ such
that
\begin{enumerate}
\item[(d)] $f(t,z)$ is bounded on $\mathbb{R}\times\mathbb{R}^p$ with
$\eta_f:=\sup_{(t,z)\in\mathbb{R}\times\mathbb{R}^p}\|f(t,z)\|<\infty$;

\item[(e)] $f(t,z)$ is Lipschitz in $z$ with a Lipschitz constant
$L_f$:
$$
\|f(t,z)-f(t,z^\prime)\|\le L_f\|z-z^\prime\|
$$
for all $t\in\mathbb{R}$ and for any $z,z^\prime\in\mathbb{R}^p$.
\end{enumerate}
For any $(A,f)\in M_{p\times p}^\vartriangle\times M_{p\times 1}$,
we will study the equations
\begin{equation}\label{eq3.1}
\dot{z}=A(t)z+f(t,z),\quad (t,z)\in\mathbb{R}\times\mathbb{R}^p
\end{equation}
and
\begin{equation}\label{eq3.2}
\dot{z}=A(t)z,\quad (t,z)\in\mathbb{R}\times\mathbb{R}^p.
\end{equation}
For any $(s,u)\in\mathbb{R}\times\mathbb{R}^p$, let $z_{A,f}(t;s,u)$
and $z_A(t;s,u)$ denote the solutions of (\ref{eq3.1}) and
(\ref{eq3.2}) with $z_{A,f}(s;s,u)=u=z_A(s;s,u)$, respectively.

The following result is important for the proof of our main theorem.

\begin{theorem}[{\cite[Theorems~3.1 and 3.2]{L74}}]\label{thm3.1}
Let $(A,f)\in M_{p\times p}^\vartriangle\times M_{p\times 1}$ be any
given. Then, there is a unique surjective mapping
$$
\Delta_{A,f}\colon\mathbb{R}\times\mathbb{R}^p\rightarrow\mathbb{R}\times\mathbb{R}^p;\
(s,u)\mapsto(s,\Delta_s(u))
$$
which possesses the following properties:
\begin{enumerate}
\item[(i)] $\Delta_{A,f}$ maps the phase-portraits of (\ref{eq3.1}) onto
that of (\ref{eq3.2}). In fact,
$$
\Delta_{A,f}(t,z_{A,f}(t;s,u))=(t,z_A(t;s,\Delta_s(u)));
$$
that is to say, the following commutativity holds:
$$
\begin{CD}
\mathbb{R}^p@>{z_{A,f}(t;s,\cdot)}>>\mathbb{R}^p\\
@V{\Delta_s}VV                 @VV{\Delta_t}V\\
\mathbb{R}^p@>{z_A(t;s,\cdot)}>>\mathbb{R}^p.
\end{CD}
$$
\item[(ii)] $\Delta_{A,f}$ is a $\varepsilon_{A,f}$-mapping, i.e.,
$\|(s,u)-\Delta_{A,f}(s,u)\|\le\varepsilon_{A,f}$ for all
$(s,u)\in\mathbb{R}\times\mathbb{R}^p$, where
$$
\varepsilon_{A,f}=\eta_f\xi_A(1+2\eta_A\xi_A)^p;
$$

\item[(iii)] For any
$(s,u),(s,u^\prime)\in\mathbb{R}\times\mathbb{R}^p$,
$z_{A,f}(t;s,u)-z_A(t;s,u^\prime)$ is bounded on $\mathbb{R}$ if and
only if $\Delta_s(u)=u^\prime$.

\item[(iv)] If
$$
L_f\le\frac{1}{\xi_A(1+\eta_A\xi_A)^p},
$$
then $\Delta_{A,f}$ is a self-homeomorphism of
$\mathbb{R}\times\mathbb{R}^p$.
\end{enumerate}
\end{theorem}

Next, we endow $M_{p\times p}^\vartriangle\times M_{p\times 1}$ with
the compact-open topology. Let $(\mathbb{P},d)$ be a metric space
with metric $d$ and $\eta_\mathbb{P}>0,\xi_\mathbb{P}>0,
L_\mathbb{P}>0$ constants with
$L_\mathbb{P}\le\frac{1}{\xi_\mathbb{P}(1+\eta_\mathbb{P}\xi_\mathbb{P})^p}$.
Let
$$
\mathfrak{S}\colon\mathbb{P}\rightarrow M_{p\times
p}^\vartriangle\times M_{p\times 1};\
\lambda\mapsto(A_\lambda,f_\lambda)
$$
be a continuous mapping such that
$\eta_{A_\lambda}\le\eta_\mathbb{P},\xi_{A_\lambda}\le\xi_\mathbb{P},L_{f_\lambda}\le
L_\mathbb{P}$, and
\begin{equation*}
\dot{z}=A_\lambda(t)z,\quad
(t,z)\in\mathbb{R}\times\mathbb{R}^p\leqno{(\ref{eq3.2})_\lambda}
\end{equation*}
has no any nontrivial bounded solutions. We consider the bounded
solutions of the equations with parameter $\lambda$
\begin{equation*}
\dot{z}=A_\lambda(t)z+f_\lambda(t,z),\quad
(t,z)\in\mathbb{R}\times\mathbb{R}^p.\leqno{(\ref{eq3.1})_\lambda}
\end{equation*}
Define
$$\Delta^*\colon\mathbb{P}\rightarrow\mathbb{R}^p$$
in the way: for any $\lambda\in\mathbb{P}$
$$
\Delta_\lambda(0,\Delta^*(\lambda))=(0,\textbf{0})\in\mathbb{R}\times\mathbb{R}^p
$$
where
$\Delta_\lambda=\Delta_{A_\lambda,f_\lambda}\colon\mathbb{R}\times\mathbb{R}^p\rightarrow\mathbb{R}\times\mathbb{R}^p$
is determined by Theorem~\ref{thm3.1} for $(\ref{eq3.1})_\lambda$
and $(\ref{eq3.2})_\lambda$.

We will need the following result, which will play a useful role in
the later proof of our main theorem in $\S\ref{sec4}$.

\begin{theorem}\label{thm3.2}
The mapping $\Delta^*\colon\mathbb{P}\rightarrow\mathbb{R}^p$ is
continuous.
\end{theorem}

\begin{proof}
Let $\lambda_0\in\mathbb{P}$ and $\varepsilon>0$. Letting
$\textbf{x}_0=\Delta^*(\lambda_0)\in\mathbb{R}^p$, we assert that
there exists some $\delta>0$ such that
$\|\Delta^*(\lambda)-\textbf{x}_0\|<\varepsilon$ whenever
$\lambda\in\mathbb{P}$ with $d(\lambda,\lambda_0)<\delta$. If the
assertion were not true, there would be a sequence
$\lambda_j\to\lambda_0$ in $\mathbb{P}$ satisfying
$\|\Delta^*(\lambda_j)-\textbf{x}_0\|\ge\varepsilon$ for all $j$.
Since for all $t\in\mathbb{R}$ we have
$$
\|z_{A_{\lambda_j},f_{\lambda_j}}(t;0,\Delta^*(\lambda_j))\|\le\eta_\mathbb{P}\xi_\mathbb{P}(1+2\eta_\mathbb{P}\xi_\mathbb{P})^p\quad
j=1,2,\ldots
$$
by Theorem~\ref{thm3.1}, we can assume
$\Delta^*(\lambda_j)\to\textbf{x}$ for some
$\textbf{x}\in\mathbb{R}^p$ and
$\|\textbf{x}-\textbf{x}_0\|\ge\varepsilon$. As $\mathfrak{S}$ is
continuous, it follows from a basic theorem of ODE that
$$
\lim_{j\to\infty}z_{A_{\lambda_j},f_{\lambda_j}}(t;0,\Delta^*(\lambda_j))=z_{A_{\lambda_0},f_{\lambda_0}}(t;0,\textbf{x})\quad\forall\,t\in\mathbb{R}
$$
which implies that $z_{A_{\lambda_0},f_{\lambda_0}}(t;0,\textbf{x})$
is a bounded solution of $(\ref{eq3.1})_{\lambda_0}$. So,
$\textbf{x}=\textbf{x}_0$, it is a contradiction.
\end{proof}
\section{Structural stability of hyperbolic sets}\label{sec4}

In this section, we will prove our main theorem stated in the
Introduction and construct an explicit example.

We assume that $S\colon\mathbb{E}\rightarrow\mathbb{R}^n$ is a
$C^1$-vector field on the $n$-dimensional Euclidean $w$-space
$\mathbb{E}, n\ge 2$, which gives rise to a flow
$\phi\colon(t,w)\mapsto t_\cdot w$. Let $\Lambda$ be a
$\phi$-invariant closed subset, not necessarily compact, of
$\mathbb{E}$ such that
\begin{enumerate}
\item[(U1)] $S^\prime(w)$ is uniformly bounded on $\Lambda$;

\item[(U2)] $0<\inf_{w\in\Lambda} \|S(w)\|\le\sup_{w\in\Lambda}
\|S(w)\|<\infty$;

\item[(U3)] $S^\prime(w)$ is uniformly continuous at $\Lambda$; that is to say, to any
$\epsilon>0$ there is some $\delta>0$ so that for any $\mathfrak
w\in\Lambda$, $\|S^\prime(w)-S^\prime(\mathfrak w)\|<\epsilon$
whenever $\|w-\mathfrak w\|<\delta$.
\end{enumerate}

Now we prove the following structural stability theorem by using
Liao methods.

\begin{theorem}\label{thm4.1}
Let $\Lambda$ be a hyperbolic set for $S$; that is to say, there
exist constants $C\ge 1, \lambda<0$ and a continuous
$\Psi$-invariant splitting
\begin{equation*}
\mathbb{T}_{w}=T_{w}^s\oplus T_{w}^u,\ \dim T_w^s=p_-(w)\quad w\in
\Lambda
\end{equation*}
such that
\begin{align*}
\|\Psi_{t_0+t,w}x\|&\le C^{-1}\exp(\lambda
t)\|\Psi_{t_0,w}x\|\quad\forall\,x\in
T_{w}^s\\
\intertext{and} \|\Psi_{t_0+t,w}x\|&\ge C\exp(-\lambda
t)\|\Psi_{t_0,w}x\|\quad\forall\,x\in T_{w}^u
\end{align*}
for any $t_0\in\mathbb{R}$ and for all $t>0$. Then for any
$\varepsilon>0$ there is a $C^1$-neighborhood $\mathcal {U}$ of $S$
in $\mathfrak{X}^1(\mathbb{E})$ such that, if $V\in\mathcal {U}$
then there exists a $\varepsilon$-topological mapping $H$ from
$\Lambda$ onto some closed subset $\Lambda_V$ which sends orbits of
$S$ in $\Lambda$ into orbits of $V$ in $\Lambda_V$, such that
$H(w)\in \Sigma_w$ and $\psi_V(t,H(w))=H(t_\cdot
w)\in\Sigma_{t_\cdot w}$ for all $w\in\Lambda$ and for any
$t\in\mathbb{R}$.
\end{theorem}

\begin{proof}
Let
$$
\mathcal
{A}=\left\{(w,\gamma)\in\mathscr{F}_{n-1}^{*\sharp}(\Lambda)\,|\,\textrm{col}_k\gamma\in
T_w^s\textrm{ for }1\le k\le p_-(w)\right\}.
$$
Clearly, $\mathcal {A}$ is a $\chi^{*\sharp}$-invariant closed
subset of $\mathscr{F}_{n-1}^{*\sharp}(\Lambda)$ with compact fibers
$\mathcal {A}_w$. For any $(w,\gamma)\in\mathcal {A}$, we consider
the reduced linearized equations
\begin{equation}\label{eq4.1}
\dot{y}=R_{w,\gamma}^*(t)y,\quad
(t,y)\in\mathbb{R}\times\mathbb{R}^{n-1},
\end{equation}
which is defined as Definition~\ref{def2.1}, and consider the
reduced standard system
\begin{equation}\label{eq4.2}
\dot{y}=R_{w,\gamma}^*(t)y+V_{\textsl{rem}(w,\gamma)}^*(t,y),\quad
(t,y)\in\mathbb{R}\times\mathbb{R}_\mathfrak c^{n-1}
\end{equation}
for any $V\in\mathfrak{X}^1(\mathbb{E})$ defined as in
Definition~\ref{def2.11}, associated with $S$. Then, we can take
from Lemma~\ref{lem2.2} a constant $\eta_\Lambda>0$ such that
$$
\sup_{t\in\mathbb{R},(w,\gamma)\in\mathcal
{A}}\|R_{w,\gamma}^*(t)\|\le \eta_\Lambda<\infty.
$$
Thus, it follows from Corollary~\ref{cor2.3} that there is another
constant $\xi_\Lambda>0$ such that
\begin{align*}
\xi_\Lambda&=\sup_{\substack{t\in\mathbb{R}\\(w,\gamma)\in\mathcal
{A}}}
\left\{\sum_{k=1}^{p_-(w)}\int_{-\infty}^te^{\int_s^t\omega_k^*(\chi^{*\sharp}(\tau,(w,\gamma)))\,d\tau}\,ds\right.\\
&{\qquad\qquad}\qquad\left.+ \sum_{k=1+p_-(w)}^{n-1}\int_t^\infty
e^{\int_s^t\omega_k^*(\chi^{*\sharp}(\tau,(w,\gamma)))\,d\tau}\,ds\right\}\\
&<\infty.
\end{align*}

By Lemma~\ref{lem2.5}(1), Theorem~\ref{thm2.13} and
Theorem~\ref{thm3.1}, there is no loss of generality in assuming
that for any $(w,\gamma)\in\mathcal {A}$ the reduced standard
systems of $S$
\begin{equation}\label{eq4.3}
\dot{y}=R_{w,\gamma}^*(t)y+S_{\textsl{rem}(w,\gamma)}^*(t,y),\quad
(t,y)\in\mathbb{R}\times\mathbb{R}_\mathfrak c^{n-1}
\end{equation}
has no any nontrivial bounded global solutions on $\mathbb{R}$.

Let $\mathfrak N_\mathfrak c(w)=\{w+x\in\Sigma_w;\,
\|x\|\le\mathfrak c\}$. Given any $\varepsilon>0$ small enough to
satisfy that for any $w\in\Lambda$ and any $w^\prime\in\Lambda$ with
$\|w-w^\prime\|<\varepsilon$, we have $t_\cdot w^\prime\in\mathfrak
N_\mathfrak c(w)$ for some
$|t|<2\varepsilon\diamondsuit_\Lambda^{-1}$, where
$\diamondsuit_\Lambda=\inf_{w\in\Lambda}\|S(w)\|>0$. On the other
hand, according to~\cite{Dai07} we may assume that for any
$w\in\Lambda$, $\mathfrak N_\mathfrak c(w)\cap\mathfrak N_\mathfrak
c(t_\cdot w)=\varnothing$ for all
$|t|<2\varepsilon\diamondsuit_\Lambda^{-1}$.

Denote
$$
\rho_\xi=\frac{\xi}{4\xi_\Lambda(1+2\eta_\Lambda\xi_\Lambda)^{n-1}}
\quad\textrm{and}\quad
\kappa=\frac{1}{4\xi_\Lambda(1+2\eta_\Lambda\xi_\Lambda)^{n-1}}
$$
for any $\xi\in(0,\mathfrak c]$. Then, by Theorem~\ref{thm2.13}
there exists a $C^1$-neighborhood $\mathcal {U}$ of $S$ in
$\mathfrak{X}^1(\mathbb{E})$ and a constant $\xi\in(0,\mathfrak c]$
with $\xi<1$ such that for any $V\in\mathcal {U}$ we have
\begin{subequations}
\begin{align}
\sup_{\substack{(w,\gamma)\in\mathcal {A}\\(t,y)\in\mathbb
R\times\mathbb
R_\xi^{n-1}}}\|V_{\textsl{rem}(w,\gamma)}^*(t,y)\|&\le\varepsilon \rho_\xi\label{eq4.4a}\\
\intertext{and}
\|V_{\textsl{rem}(w,\gamma)}^*(t,y)-V_{\textsl{rem}(w,\gamma)}^*(t,y^\prime)\|&\le
\kappa\|y-y^\prime\|\quad\forall\,y,y^\prime\in\mathbb
R_{\xi}^{n-1}\label{eq4.4b}
\end{align}
\end{subequations}
uniformly for $(t,(w,\gamma))\in\mathbb R\times\mathcal {A}$.

Fix some $C^\infty$ bump function
$b\colon[0,\infty)\rightarrow[0,1]$ with $b|[0,1/2]\equiv 1$ and
$b|[1,\infty)\equiv 0$. For any $V\in\mathcal {U}$ and any
$(w,\gamma)\in\mathcal {A}$, let
$$
\widetilde{V}_{\textsl{rem}(w,\gamma),\xi}(t,y)=\begin{cases}
b(\|y\|/\xi)V_{\textsl{rem}(w,\gamma)}^*(t,y) &\textrm{ for }
\|y\|\le\mathfrak c,\\
\textbf{0} &\textrm{ for }\|y\|\ge \mathfrak c.
\end{cases}
$$
Next, we consider the adapted differential equations
\begin{equation}\label{eq4.5}
\dot{y}=R_{w,\gamma}^*(t)y+\widetilde{V}_{\textsl{rem}(w,\gamma),\xi}(t,y),\quad
(t,y)\in\mathbb{R}\times\mathbb{R}^{n-1}.
\end{equation}
It is easily seen that $R_{w,\gamma}^*(t)\in
M_{(n-1)\times(n-1)}^\vartriangle$ and
$\widetilde{V}_{\textsl{rem}(w,\gamma),\xi}(t,y)\in M_{(n-1)\times
1}$ as in $\S3$ in the case $p=n-1$ and $p_-=p_-(w)$ for any
$(w,\gamma)\in\mathcal {A}$. Let $y_{V,(w,\gamma),\xi}(t;s,u)$ be
the solution of (\ref{eq4.5}) such that
$y_{V,(w,\gamma),\xi}(s;s,u)=u$ for any
$(s,u)\in\mathbb{R}\times\mathbb{R}^{n-1}$.

Given any $V\in\mathcal {U}$.

For any $(w,\gamma)\in\mathcal {A}$, it follows from
Theorem~\ref{thm3.1} that there \emph{uniquely} corresponds an
$\textbf{x}\in\mathbb{R}^{n-1}$, writing $h_{V,\xi}(w)=\gamma
\textbf{x}\in\Sigma_w-w=\mathbb{T}_w$, such that
$y_{V,(w,\gamma),\xi}(t;0,\textbf{x})$ is bounded on $\mathbb{R}$
with
\begin{equation}\label{eq4.6}
{\sup}_{t\in\mathbb{R}}\|y_{V,(w,\gamma),\xi}(t;0,\textbf{x})\|\le\varepsilon\xi/4.
\end{equation}
So, $y_{V,(w,\gamma),\xi}(t;0,\textbf{x})$ is also the solution of
(\ref{eq4.2}). According to Theorem~\ref{thm2.12}(2) we easily see
that such $h_{V,\xi}(w)$ is independent of the choice of $\gamma$ in
$\mathcal {A}_w$ and is such that
$\|h_{V,\xi}(w)\|\le\min\{\varepsilon,\xi\}/4$ for $w\in\Lambda$. By
Theorem~\ref{thm3.1} and Theorem~\ref{thm2.12}(2) again we have
easily
\begin{equation}\label{eq4.7}
h_{V,\xi}(t_\cdot
w)=\psi_{V;t,w}(h_{V,\xi}(w))\in\mathbb{T}_{t_\cdot w}\quad
\forall\,t\in\mathbb{R},
\end{equation}
since $\psi_{V;t,w}(h_{V,\xi}(w))=\widehat{\mathcal
{P}}_{w,\gamma}^*(t,y_{V,(w,\gamma),\xi}(t;0,\textbf{x}))$ for any
$w\in\Lambda$. Moreover, we can assert that the mapping $w\mapsto
w+h_{V,\xi}(w)$ is injective. In fact, if
$w+h_{V,\xi}(w)=w^\prime+h_{V,\xi}(w^\prime)$ for some
$w,w^\prime\in\Lambda$, then $\|t_\cdot w-t_\cdot
w^\prime\|\le\varepsilon/2$ for all $t\in\mathbb{R}$. Since
(\ref{eq4.3}) has only one global bounded solution on $\mathbb{R}$,
there is some $t^\prime$ with
$|t^\prime|<2\varepsilon\diamondsuit_\Lambda^{-1}$ such that
$t^\prime_\cdot w^\prime=w$. Thus, $t^\prime=0$. Otherwise
$\mathfrak N_\mathfrak c(w)\cap\mathfrak N_\mathfrak
c(w^\prime)\not=\varnothing$, it is a contradiction.

Let $\Lambda_V=\{w+h_{V,\xi}(w)\,|\,w\in\Lambda\}$ and
$H_V\colon\Lambda\rightarrow\Lambda_V;\, w\mapsto
w+h_{V,\xi}(w)\in\Sigma_w$. Clearly, $\|w-H_V(w)\|<\varepsilon$. It
remains to prove that $\Lambda_V$ is closed in $\mathbb{E}$ and
$H_V$ is a homeomorphism.

At first, we show that $\Lambda_V$ is closed in $\mathbb{E}$ and
$H_V^{-1}\colon\Lambda_V\rightarrow\Lambda$ continuous as well. Let
$w_j^\prime\to w^\prime$ with $w_j^\prime\in\Lambda_V$ and
$w_j=H_V^{-1}(w_j^\prime)$ for $j=1,2,\ldots$. We have to prove
$w_j\to w$ for some $w\in\Lambda$ and $w^\prime=H_V(w)$. By the
definition of $H_V$, there is a sequence $(w_j,\gamma_j)$ in
$\mathcal {A}$ and a sequence $(\textbf{x}_j)$ in $\mathbb{R}^{n-1}$
such that
$$
w_j^\prime=w_j+h_{V,\xi}(w_j)=w_j+\gamma_j\textbf{x}_j\quad
\textrm{for } j=1,2,\ldots.
$$
Since $\gamma_j\in\mathcal
{A}_{w_j}\subset\mathscr{F}_{n-1,w_j}^{*\sharp}\subset\mathscr{F}_{n-1}^\sharp,
\|\textbf{x}_j\|\le\varepsilon\xi/4$ and $\mathscr{F}_{n-1}^\sharp$
is compact, without loss of generality we may assume that
$\gamma_j\to\gamma$ in $\mathscr{F}_{n-1}^\sharp$ and
$\textbf{x}_j\to \textbf{x}$ in $\mathbb{R}^{n-1}$. Let
$w=w^\prime-\gamma \textbf{x}$. Then $w_j\to w$ in $\Lambda$ and
$w^\prime=w+\gamma\textbf{x}$ and $(w,\gamma)\in\mathcal {A}$. In
order to prove $w^\prime=H_V(w)$, it is sufficient to prove that
$h_{V,\xi}(w)=\gamma\textbf{x}$. In fact, from
Theorem~\ref{thm2.12}(1) and a basic theorem of ODE, we have
$$
\lim_{j\to\infty}\sup_{|t|<T,\|u\|\le\mathfrak
c}\|y_{V,(w_j,\gamma_j),\xi}(t;0,u)-y_{V,(w,\gamma),\xi}(t;0,u)\|=0\quad\forall\,T>0.
$$
Thus, for all $t\in\mathbb{R}$ we have
$$
\lim_{j\to\infty}y_{V,(w_j,\gamma_j),\xi}(t;0,\textbf{x}_j)=y_{V,(w,\gamma),\xi}(t;0,\textbf{x}),
$$
which means
$\|y_{V,(w,\gamma),\xi}(t;0,\textbf{x})\|\le\varepsilon\xi/4$ for
all $t\in\mathbb{R}$. So, $h_{V,\xi}(w)=\gamma\textbf{x}$, as
desired.

We can show that $H_V$ is continuous by Theorem~\ref{thm3.2}.

Thus, the theorem is proved.
\end{proof}

\begin{rem}
If $\Lambda=\mathbb{E}$, then $\Lambda_V=\mathbb{E}$ by a standard
topology argument. Indeed, letting
$S^{n+1}=\mathbb{E}\cup\{\infty\}$, $H_V$ has a continuous extension
from the topological sphere $S^{n+1}$ to itself which maps $\infty$
to $\infty$ and is homotopic to the identity. Thus, from
differential topology we know that $H_V(\mathbb{E})=\mathbb{E}$.
\end{rem}

We conclude our arguments with an example.

\begin{example}\label{exa4.3}
Let $S(x,y,z)=(1,y,-z)^{\mathrm T}\in\mathbb{R}^3$ for any
$(x,y,z)\in\mathbb{E}^3$, which is a differential system on the
3-dimensional Euclidean $(x,y,z)$-space $\mathbb{E}^3$. Let
$\Lambda=\mathbb{R}\times\{0\}\times\{0\}$. Then, $S$ gives rise to
the $C^1$-flow $\phi\colon(t,(x,y,z))\mapsto(x+t,ye^t,ze^{-t})$, and
$S$ is hyperbolic with $\mathbb{T}_{(x,0,0)}=T_{(x,0,0)}^s\oplus
T_{(x,0,0)}^u$, where
$T_{(x,0,0)}^s=\{0\}\times\{0\}\times\mathbb{R},
T_{(x,0,0)}^u=\{0\}\times\mathbb{R}\times\{0\}$ for any
$(x,0,0)\in\Lambda$, and $S$ satisfies conditions (U1), (U2) and
(U3) on $\Lambda$. Thus, $S$ is structurally stable on $\Lambda$
from the Main Theorem. Particularly, for any $\varepsilon>0$, if
$V\in\mathfrak{X}^1(\mathbb{E}^3)$ is $C^1$-close to $S$, then $V$
has an integral curve which lies in the $\varepsilon$-tubular
neighborhood of $\mathbb{R}\times\{(0,0)\}$.
\end{example}

\end{document}